\documentclass[12pt]{article}
\usepackage{amsmath}
\usepackage{float}
\usepackage{amsthm}
\usepackage{amssymb}
\usepackage[utf8x]{inputenc}
\usepackage{caption}
\usepackage{graphics}
\usepackage{enumerate} 
\usepackage{algorithmic}
\usepackage[Algorithm,ruled]{algorithm}
\usepackage{titlesec}
\usepackage{sectsty}
\titleformat*{\section}{\LARGE\bfseries}
\titleformat*{\subsection}{\Large\bfseries}
\titleformat*{\subsubsection}{\large\bfseries}
\sectionfont{\Huge}
\subsectionfont{\large}
\subsubsectionfont{\normalsize}
\usepackage{multicol}
\usepackage{lmodern}
\usepackage{marvosym} 
\usepackage{lipsum}
\usepackage{mwe}
\usepackage{caption}
\usepackage{subcaption} 
\newtheoremstyle{case}{}{}{}{}{}{:}{ }{}
\theoremstyle{case}

\newcommand{\be}{\begin{equation}}
\newcommand{\ee}{\end{equation}}
\newcommand{\ben}{\begin{eqnarray*}}
\newcommand{\een}{\end{eqnarray*}}
\newtheorem{examp}{\sc Example}
\newtheorem{remk}{\sc Remark}
\newtheorem{corol}{\sc Corollary}
\newtheorem{lemma}{\sc lemma}
\newtheorem{theorem}{\sc theorem}
\newtheorem{defn}{\sc definition}
\newcommand{\bt}{\begin{theorem}}
\newcommand{\et}{\end{theorem}}
\newcommand{\bl}{\begin{lemma}}
\newcommand{\el}{\end{lemma}}
\newcommand{\bed}{\begin{defn}}
\newcommand{\eed}{\end{defn}}
\newcommand{\brem}{\begin{remk}}
\newcommand{\erem}{\end{remk}}
\newcommand{\bex}{\begin{examp}}
\newcommand{\eex}{\end{examp}}
\newcommand{\bcl}{\begin{corol}}
\newcommand{\ecl}{\end{corol}}

\newcommand{\NI}{\noindent}

\newcommand{\vsp}{\vskip 0.5em}

\theoremstyle{definition}
\theoremstyle{remark}

\numberwithin{equation}{section}
\numberwithin{theorem}{section}
\numberwithin{lemma}{section}

\begin{document}
\title{\large {\bf{\sc Solution Of Tensor Complementarity Problem Using Homotopy Function}}}
\author{ A. Dutta$^{a, 1}$ , Bharat Kumar$^{b, 2}$, Deepmala$^{c, 3}$ and A. K. Das$^{d, 4}$\\
\emph{\small $^{a}$Department of Mathematics, Jadavpur University, Kolkata, 700 032, India}\\	
\emph{\small $^{b,c}$Department of Mathematics, PDPM-Indian Institute of Information Technology}\\
\emph{\small Design and Manufacturing, Jabalpur - 482005 (MP), India, India}\\
\emph{\small $^{d}$SQC \& OR Unit, Indian Statistical Institute, Kolkata, 700 108, India}\\
\emph{\small $^{1}$Email: aritradutta001@gmail.com}\\
\emph{\small $^{2}$Email: bharatnishad.kanpu@gmail.com}\\
\emph{\small $^{3}$Email: dmrai23@gmail.com}\\
\emph{\small $^{4}$Email: akdas@isical.ac.in} \\
 }
%\author{ \\
%\emph{\small $^{a}$Department of Mathematics, Jadavpur University, %Kolkata, 700 032, India}\\	
%\emph{\small $^{b}$SQC \& OR Unit, Indian Statistical Institute, %Kolkata, 700 108, India}\\
%\emph{\small $^{1}$Email: aritradutta001@gmail.com}\\
%\emph{\small $^{2}$Email: akdas@isical.ac.in} \\
%}
\date{}

\maketitle

\begin{abstract}
The paper aims to propose a suitable method in finding the solution of tensor complementarity problem. The tensor complementarity problem is a subclass of nonlinear complementarity problems for which the involved function is defined by a tensor. We propose a new homotopy function with smooth and bounded homotopy path to obtain  solution of the tensor complementarity problem under some conditions. A homotopy continuation method is developed based on the proposed homotopy function. Several numerical examples are provided to show the effectiveness of the proposed homotopy continuation method.
\vsp
\NI{\bf Keywords:} Tensor complementarity problem, homotopy function, homotopy path, bounded smooth curve, semipositive tensor, linear complementarity problem. \\

\NI{\bf AMS subject classifications:} 90C33, 90C30, 15A69.
\end{abstract}
\section{Introduction}

During last several decades, researchers are of keen interest to study the complementarity theory finding  solution of linear complementarity problem. The linear programming problem, linear fractional programming problem, convex quadratic programming problem, bimatrix game problem and quadratic multiobjective programming problem are some of the optimization problems which can be modelled to be linear complementarity problem. For details
see \cite{das2017finiteness}, \cite{article12}, \cite{article11} and \cite{article03} and references cited therein. Several matrix classes are introduced for the purpose of study of theoretical properties, applications and solution methods of complementarity problem. For details see \cite{jana2019hidden}, \cite{jana2021more}, \cite{article1}, \cite{mohan2001more}, \cite{neogy2013weak},\cite{article07}, \cite{das201} and \cite{neogy2005almost} and references cited therein. The linear complementary problem has a key role to obtain the value vector and optimal stationary strategies for discounted and undiscounded zero-sum stochastic games. For details see \cite{mondal2016discounted}, \cite{neogy2008mixture} and \cite{neogy2005linear}. The complementarity problem establishes a vital connection with multiobjective programming problem for its weighted problem and the solution point \cite{article78}. The complementarity problems are considered with respect to principal pivot transforms and pivotal method to its solution point of view. See \cite{das2016properties}, \cite{neogy2012generalized}, \cite{das2019convergence} and \cite{neogy2005principal}. The linear complementarity problem arising from a free boundary problem can be  reformulated as a fixed-point equation. Zhang \cite{zhang2021modified} presented a modified modulus-based multigrid method to solve this fixed-point equation. 
Recently there has been renewed interest in finding the solution of tensor complementarity problem. The tensor complementarity problem is considered to be a subclass of nonlinear complementarity problems with special class of polynomials using tensor.  The function involved in the tensor complementarity problem is more complex than that of the function involving matrix. It is of interest to develop numerical methods to solve tensor complementarity problem. The tensor complementarity problem has several applications. For details see \cite{du2019mixed}, \cite{zhang2019potential}, \cite{han2019continuation}, \cite{ye1997wiley},\cite{qi2019tensor}, \cite{xie2017iterative}, \cite{xu2019equivalent}, \cite{luo2017sparsest}, \cite{liu2018tensor}, \cite{huang2019tensor}, \cite{huang2017formulating}.

The purpose of the study is to develop a suitable homotopy function to find the solution of tensor complementarity problem. The basic idea of homotopy method is to construct a homotopy continuation path 
from the auxiliary mapping $g$ to the object mapping $f$. 
Suppose the given problem is to find a root of the non-linear equation $f(x) = 0$
and suppose $g(x) = 0$ is an auxiliary equation with $g(x_0)=0$. Then the
homotopy function $H:\mathbb{R}^{n+1} \to \mathbb{R}^n$ can be
 defined as $H(x, \mu) = $
 $ (1-\mu)f(x) + \mu g(x),$
 $ 0 \leq \mu \leq 1.$
Then we consider the homotopy equation $H(x, \mu) = 0,$ where $(x_0,1)$ is a known solution of the homotopy equation. Our aim is to find the solution of the equation $f(x)=0$ from the known solution of $g(x) = 0$ by solving the homotopy equation $H(x, \mu) = 0$ varrying the values of $\mu$  from $1$ to $0$. Kojima et al showed that under some conditions nonlinear complementarity problem can be solvable by homotopy continuation method. For details see \cite{kojima1991homotopy},  \cite{kojima1994global}, \cite{xuuu}.
 Han \cite{han2019continuation} introduced a homotopy continuation method for solving tensor complementarity problems. Han \cite{han2017homotopy}, Yan et al. \cite{yan2022homotopy} solve multilinear systems with  strong completely positive tensors by homotopy method.
\vsp
The paper is organized as follows. Section 2 presents some basic notations and results. In section 3, we propose a new homotopy function to find the solution of tensor complementarity problem. We construct a smooth and bounded homotopy path  obtaining the solution of the tensor complementarity problem under some conditions as the homotopy parameter $\mu$ tends to $0$. We prove necessary and sufficient conditions to obtain the solution of tensor complementarity problem from the solution of the homotopy equation. We also find the sign of the positive tangent direction of the homotopy path. We use a modified interior-point bounded homotopy path algorithm for solving the linear complementarity problem in section 4. Finally, in section 4, some numerical results are given to illustrate the effectiveness of the homotopy function.

\section{Preliminaries}
We denote $[n]= \{1, 2, \cdots, n\}$. We denote the $n$ dimensional real space by $\mathbb{R}^n$ where $\mathbb{R}^n_+$ and $\mathbb{R}^n _{++}$ denote the nonnegative and positive orthant of $\mathbb{R}^n.$  Any vector $x\in \mathbb{R}^n$ is a column vector and  $x^{t}$ denotes the row transpose of $x.$ $e$ denotes the vector of all $1.$ An $m$th order $n$th dimensional tensor $\mathcal{A} = (a_{i_1 i_2 \cdots i_m})$ is a multi-array of entries $a_{i_1 i_2 \cdots i_m}$, where $i_j \in [n]$ for $j \in [m]$. The set of all $m$th order, $n$th dimensional tensors is denoted by $ T_{m,n}$. Here we consider vectors, matrices and tensors with real entries. 

\NI For $\mathcal{A}\in T_{m,n} $ and $x\in \mathbb{R}^n,~ \mathcal{A}x^{m-2}\in \mathbb{R}^{n\times n} $ is a matrix defined by $(\mathcal{A}x^{m-2})_{ij} = \sum_{ i_3, ...i_m =1}^{n} a_{ij i_3 \cdots i_m} x_{i_3} x_{i_4} \cdots x_{i_m} , ~~~\mbox{for all}~i,j \in [n], $  $\mathcal{A}x^{m-1}\in \mathbb{R}^n $ is a vector defined by $ (\mathcal{A}x^{m-1})_i = \sum_{i_2, i_3,\cdots i_m =1}^{n} a_{i i_2 i_3 \cdots i_m} x_{i_2} x_{i_3} \cdots x_{i_m} , ~~~\mbox{for all}~i \in [n] $ and $\mathcal{A}x^m\in \mathbb{R} $ is a scalar defined by
$ x^T \mathcal{A}x^{m-1} = \mathcal{A}x^m = \sum_{i_1,i_2, i_3, ...i_m =1}^{n} a_{i_1 i_2 i_3 ...i_m} x_{i_1} x_{i_2} \cdots x_{i_m} .$

Now we consider the tensor $\mathcal{A}\in T_{m,n}$ and $ q\in {R}^n$ then the problem  is to find $x$ such that 
\begin{equation}\label{TCP}
    x\geq 0, ~~~\omega=\mathcal{A}x^{m-1} + q\geq 0, ~~~\mbox{and}~~ x^{T}(\mathcal{A}x^{m-1}+q)=0
\end{equation}
 is called a tensor complementarity problem, denoted by the TCP$(q,\mathcal{A})$.\\

\NI We define the partially symmetrized tensor of a tensor $\mathcal{A}=(a_{i_1 i_2 \cdots i_m})$ with respect to the indices $i_2,i_3,\cdots,i_m$ is defined by $\hat{\mathcal{A}}=\hat{A}_{i_1 i_2 \cdots i_m}=\frac{1}{(m-1)! }\sum_{\pi} a_{i_1 \pi(i_2 i_3 \cdots i_m)}$, where the sum is over all the permutations $\pi(i_2 i_3 \cdots i_m).$ For details see \cite{han2017homotopy}.
The partial derivative matrix of $\mathcal{A} x^{m-1}$ with respect to $x$ is
$D_x \mathcal{A}x^{m-1}=(m-1)\hat{\mathcal{A}}x^{m-2}$.
 Note that $\mathcal{A}x^{m-1}=\hat{\mathcal{A}}x^{m-1}  \forall \ x \in \mathbb{R}^n$. 
\vsp
Now we state some results which will be required in the next section.
\begin{lemma}
	(Generalizations of Sard's Theorem\cite{chow1978finding}) \ Let $U \subset \mathbb{R}^n$ be an open set and $f :\mathbb{R}^n \to \mathbb{R}^p$ be smooth. We say $y \in \mathbb{R}^p$ is a regular value for $f$ if $\text{Range} Df(x) = \mathbb{R}^p $ $\forall x \in f^{-1}(y),$ where $Df(x)$ denotes the $n \times p$ matrix of partial derivatives of $f(x).$
\end{lemma}
\begin{lemma}\label{par}
	(Parameterized Sard Theorem \cite{Wang}) \ Let $V \subset \mathbb{R}^n, U \subset \mathbb{R}^m$ be open sets, and let $\phi:V\times U \to \mathbb{R}^k$ be a $C^\alpha$ mapping, where $\alpha >\text{max}\{0,m-k\}.$ If $0\in \mathbb{R}^k$ is a regular value of $\phi,$ then for almost all $a \in V, 0$ is a regular value of $\phi _ a=\phi(a,.).$    
\end{lemma}
\begin{lemma}\label{inv}
	(The inverse image theorem \cite{Wang}) \ Let $\phi : U \subset \mathbb{R}^n \to \mathbb{R}^p$ be $C^\alpha$ mapping, where $\alpha >\text{max}\{0,n-p\}.$ Then $\phi^{-1}(0)$ consists of some $(n-p)$-dimensional $C^\alpha$ manifolds. 
\end{lemma}
\begin{lemma}\label{cl}
	(Classification theorem of one-dimensional smooth manifold \cite{N}) \ One-dimensional smooth manifold is diffeomorphic to a unit circle or a unit interval.
\end{lemma}
\section{Main Results}
In 2019 Han  \cite{han2019continuation} proposed the following homotopy function to solve the TCP$(\mathcal{A},q)$ where $\mathcal{A}$ is a \textit{ strong strictly semipositive tensor}.
We begin by the followings:
\begin{center}
$\mathcal{F}=\{(x,w)\in \mathbb{R}^{2n}:x>0,w>0\}$ \\ 
$\mathcal{\bar{F}}=\{(x,w)\in \mathbb{R}^{2n}:x\geq 0,w\geq 0\}$\\
$\mathcal{F}_1=\mathcal{F} \times R_{++}^n \times R_{++}^n$ \\
$\mathcal{\bar{F}}_1=\mathcal{\bar{F}} \times R_{+}^n \times R_{+}^n.$\\ $\partial{\mathcal{F}_1}$ denotes the boundary of $\bar{\mathcal{F}_1}.$ \\
\end{center}

 Now we propose the following homotopy function 
 \begin{equation}\label{homf}
 H(v,v^{(0)},\mu)=\left[\begin{array}{c} 
 (1-\mu)(w-z_1+(m-1)(\hat{\mathcal{A}}x^{m-2})^T(x-z_2))+\mu(x-x^{(0)}) \\
 Z_1x-\mu Z_1^{(0)}x^{(0)}\\
 Z_2w-\mu Z_2^{(0)}w^{(0)}+(1-\mu)Xw\\
 w-(1-\mu)(\mathcal{A}x^{m-1} + q)-\mu w^{(0)}\\
  \end{array}\right]=0
 \end{equation}
 where $\hat{\mathcal{A}}$ is partially symmetrize tensor, \ $e=[1,1,\cdots,1]^T; \  X=\text{diag}(x);\ Z_1=\text{diag}(z_1); \ Z_2=\text{diag}(z_2); \ Z_1^{(0)}=\text{diag}(z_1^{(0)}); \ Z_2^{(0)}=\text{diag}(z_2^{(0)}); \ v=(x,w,z_1,z_2)\in \mathcal{\bar{F}}_1; \ v^{(0)}=(x^{(0)},w^{(0)},{z_1}^{(0)},{z_2}^{(0)})\in \mathcal{F}_1; \ \mu \in (0,1]$.  
\vsp
We prove that the proposed homotopy function contains a smooth and bounded path.

\begin{theorem}\label{reg}
	For almost all initial points $v^{(0)}\in \mathcal{F}_1,$ $0$ is a regular value of the homotopy function $H:\mathbb{R}^{4n} \times (0,1] \to \mathbb{R}^{4n}$ and the zero point set $H_{v^{(0)}} ^{-1}(0)=\{(v,\mu)\in \mathcal{F}_1 \times (0,1]:H(v,v^{(0)},\mu)=0\}$ contains a smooth curve $\Gamma_v^{(0)}$ starting from $(v^{(0)},1).$  
\end{theorem}
\begin{proof}
	The jacobian matrix of the above homotopy function $H(v,v^{(0)},\mu)$ is denoted by $DH(v,v^{(0)},\mu))$ and we have $DH(v,v^{(0)},\mu))=$ $\left[\begin{array}{ccc} 
	\frac{\partial{H(v,v^{(0)},\mu)}}{\partial{v}} & 	\frac{\partial{H(v,v^{(0)},\mu)}}{\partial{v^{(0)}}} & \frac{\partial{H(v,v^{(0)},\mu)}}{\partial{\mu}}\\ 
	\end{array}\right].$ For all $v^{(0)} \in \mathcal{F}_1$ and $\mu \in (0,1],$ we have\\ $\frac{\partial{H(v,v^{(0)},\mu)}}{\partial{v^{(0)}}}=$ $\begin{bmatrix} -\mu I & 0 & 0 & 0\\
    -\mu Z_1^{(0)} & 0 & -\mu X^{(0)} & 0\\
     0 & -\mu Z_2^{(0)} & 0 & -\mu W^{(0)}\\
      0 & -\mu I & 0 & 0\\
      \end{bmatrix}=L_1$,\\ 
       where $W^{(0)}=\text{diag}(w^{(0)}), X^{(0)}=\text{diag}(x^{(0)})$, $Z_1^{(0)}=\text{diag}(z_1^{(0)})$, $Z_2^{(0)}=\text{diag}(z_2^{(0)})$.\\ Note that after some elementary row operations the block matrix $L_1$ becomes a lower triangular matrix with nonzero diagonals. Now
	$\text{det}(\frac{\partial{H}}{ \partial{v^{(0)}}})=\mu^{4n}\prod_{i=1}^{n_1} x_i^{(0)}w_i^{(0)}\neq 0$ for $\mu \in (0,1].$  
	Thus $DH(v,v^{(0)},\mu))$ is of full row rank. Therefore, $0$ is a regular value of $H(v,v^{(0)},\mu)).$ By Lemmas \ref{par} and \ref{inv}, for almost all  $v^{(0)} \in \mathcal{F}_1,$ $0$ is a regular value of $H(v,v^{(0)},\mu)$ and $H_{v^{(0)}} ^{-1}(0)$ consists of some smooth curves and $H(v^{(0)},v^{(0)},1)=0.$ Hence there must be a smooth curve  $\Gamma_v^{(0)}$ starting from  $(v^{(0)},1).$
\end{proof}

\begin{theorem}\label{bdd}
	Let $\mathcal{F}_1$ be a nonempty set. Assume that there exists a sequence of points $\{m^k\} \subset \Gamma_v^{(0)} \subset \mathcal{F}_1 \times (0,1],$ where $m^k=(x^k,w^k,z_1^k,z_2^k, \mu^k)$ such that $\|x^k\|< \infty \ \text{as} \ k \to \infty$ and $\|z_2 ^k\|< \infty \ \text{as} \ k \to \infty$. For a given $v^{(0)} \in \mathcal{F}_1,$ if  $0$ is a regular value of $H(v,v^{(0)},\mu),$ then $\Gamma_v^{(0)}$ is a bounded curve in $\mathcal{\bar{F}}_1 \times (0,1].$ 
\end{theorem}
\begin{proof}
		We have that $0$ is a regular value of $H(v,v^{(0)},\mu)$  by theorem \ref{reg} and $\mathcal{F}_1$ be a nonempty set. It is clear that the set $(0,1]$ is bounded. So there exists a sequence of points $\{x^k,w^k,z_1 ^k,z_2 ^k,\mu^k\} \subset \Gamma_v^{(0)} \times (0,1],$ such that $\lim\limits_{k \to \infty}x^k=\bar{x}, \lim\limits_{k \to \infty}w^k=\bar{w},  \lim\limits_{k \to \infty}z_2 ^k=\bar{z_2},  \lim\limits_{k \to \infty}\mu^k=\bar{\mu}.$ Note that the boundedness of the sequence $\{x^k\}$ gurantees the boundedness of the sequence $\{w^k\}$.
	 By contradiction we assume that $\Gamma_v^{(0)} \subset \mathcal{F}_1 \times (0,1]$ is an unbounded curve. Then there exists a sequence of points $\{m^k\},$ where $m^k=(v^k, \mu^k) \subset \Gamma_v^{(0)}$ such that $\|(v^k, \mu^k)\| \to \infty.$ Assume that  $ \|z_1 ^k\| \to \infty \ \text{as} \ k \to \infty.$  Since $\Gamma_v^{(0)} \subset H^{-1}(0),$ we have
\begin{equation}\label{zzq}
(1-\mu^k)[w^k-z_1 ^k+(m-1)(\hat{\mathcal{A}}(x^k)^{m-2})^T(x^k-z_2 ^k)]+\mu^k(x^k-x^{(0)}) =0 
\end{equation}	
\begin{equation}\label{yyq}
 Z_1 ^k x^k-\mu^k Z_1^{(0)}x^{(0)}=0
\end{equation}
\begin{equation}\label{wwq}
 Z_2 ^k w^k-\mu^k Z_2^{(0)}w^{(0)}+(1-\mu^k)X^k w^k=0
\end{equation}
\begin{equation}\label{xxq}
    w^k-(1-\mu^k)(\mathcal{A}(x^k)^{m-1} + q)-\mu^k w^{(0)}=0 
\end{equation}
where $Z_1 ^k=\text{diag}(z_1 ^k), X^k=\text{diag}(x^k)$, $W^k=\text{diag}(w^k)$ and $Z_2 ^k=\text{diag}(z_2 ^k).$\\
\vsp

\NI Let
$\bar{\mu} \in [0,1], \|z_1 ^k\|=\infty$ and $\|z_2 ^k\|<\infty$ as $k \to \infty.$
Then $\exists \ i \in \{1,2,\cdots, n\}$ such that $z_{1i} ^k \to \infty$ as $k \to \infty.$ Let $I_{1z}=\{i\in\{1,2,\cdots n\} : \lim\limits_{k\to \infty}z_{1i} ^k = \infty\}.$ For $\bar{\mu} \in [0,1)$ and $i\in I_{1z}$, we obtain from equation \ref{zzq},\\ $(1-\mu^k)[w^k-z_1 ^k+(m-1)(\hat{\mathcal{A}}(x^k)^{m-2})^T(x^k-z_2 ^k)]+\mu^k(x^k-x^{(0)}) =0$\\
$\implies (1-\mu^k)z_{1i} ^k=(1-\mu^k)[w_i ^k+(m-1)(\hat{\mathcal{A}}(x^k)^{m-2})_i^T(x_i ^k-z_{2i} ^k)]+\mu^k(x_i ^k-x_i ^{(0)})\\
\implies z_{1i} ^k=[w_i ^k+(m-1)(\hat{\mathcal{A}}(x^k)^{m-2})_i^T(x_i ^k-z_{2i} ^k)]+\frac{\mu^k}{(1-\mu^k)}(x_i ^k-x_i ^{(0)}).$\\ As $k \to \infty$ right hand side is bounded, but left hand side is unbounded. It contradicts that $\|z_1 ^k\|=\infty.$\\ When $\bar{\mu}=1,$  from equation \ref{yyq}, we obtain, $x_i ^k=\frac{\mu^k z_{1i} ^{(0)} x_i ^{(0)}}{z_{1i} ^k}$ for $i \in I_{1z}.$ As $k \to \infty, x_i ^k \to 0.$ \\ Again from equation \ref{zzq}, we obtain	$x_i ^{(0)}=\frac{(1-\mu^k)}{\mu^k}[w_i ^k-z_{1i} ^k+(m-1)(\hat{\mathcal{A}}(x^k)^{m-2})_i ^T(x_i ^k-z_{2i} ^k)]+x_i ^k$ for $i \in I_{1z}.$ As $k \to \infty,$ we have  $x_i ^{(0)}=-\lim\limits_{k\to \infty}\frac{(1-\mu^k)}{\mu^k}z_{1i} ^k \leq 0.$ It contradicts that $\|z_1 ^k\|=\infty.$  So  $\Gamma_v^{(0)}$ is a bounded curve in $\mathcal{F}_1 \times (0,1].$  
\end{proof}
Therefore the boundedness of the sequences $\{x_k\}$ and $\{z_2 ^k\}$ gurantee the boundedness of the sequence $\{z_1 ^k\},$ i.e. the boundedness of the sequence $\{m^k\}.$  Now we prove the boundedness of the sequences  $\{x_k\}$ and $\{z_2 ^k\}$.

\begin{theorem}\label{001}
 For a given $v^{(0)} \in \mathcal{F}_1,$ if  $0$ is a regular value of $H(v,v^{(0)},\mu),$ then $\Gamma_v^{(0)}$ is a bounded curve in $\mathcal{\bar{F}}_1 \times (0,1].$ 

\end{theorem}
\begin{proof}
		 Suppose the solution set $\Gamma_v^{(0)}$ is unbounded for $\mu\in [0,1)$. Then there exists a sequence of points $\{m^k\} \subset \Gamma_v^{(0)} \subset \mathcal{F}_1 \times [0,1),$ where $m^k=(v^k,\mu^k)=(x^k,w^k,z_1 ^k,z_2 ^k, \mu^k)$ such that  $\lim\limits_{k\to \infty}\mu^k=\bar{\mu} \in [0,1)$.  There also exist $(\xi, \zeta, \eta, \sigma) \in R_+^{4n}$ such that $e^T \xi =1.$ Now we consider following two cases. \\
	\vsp
	\NI Case 1:  $\|z_2^k\|<\infty$ as $k \to \infty$. Since the solution set $\Gamma_v^{(0)}$ is unbounded, we consider the following two subcases.\\ \NI Subcase (i) $\lim\limits_{k\to \infty}e^Tx^k=\infty:$\\   Let $\lim\limits_{k\to \infty}\frac{x^k}{e^Tx^k}=\xi \geq 0 ,$ $\lim\limits_{k\to \infty}\frac{w^k}{e^Tx^k}=\zeta \geq 0 $ and $\lim_{k\to  \infty}\frac{z_1^k}{e^Tx^k}=\eta \geq 0. $ So it is clear that $e^T\xi=1.$ Multiplying $(x^k)^T $ in both sides of \ref{zzq} and taking limit $k \to \infty $ from equation \ref{zzq}, we write
	$(1-\mu^k)[(x^k)^Tw^k-(x^k)^Tz_1 ^k+(m-1)(x^k)^T(\hat{\mathcal{A}}(x^k)^{m-2})^T(x^k-z_2 ^k)]+\mu^k(x^k)^T(x^k-x^{(0)}) =0 .$ This implies 
	\begin{equation}\label{zq}
	    (1-\mu^k)[(x^k)^Tw^k-(x^k)^Tz_1 ^k+(m-1)(\hat{\mathcal{A}}(x^k)^{m-1})^T(x^k-z_2 ^k)]+\mu^k(x^k)^T(x^k-x^{(0)})=0.
	 	\end{equation}
	Dividing the equations \ref{zq}, \ \ref{yyq} and \ref{wwq} by $(e^Tx^k)^2$ and taking limit $k \to \infty $  and dividing by  $(e^Tx^k)$ and taking limit $k \to \infty $ from equation \ref{xxq}, we write\\
	 \begin{eqnarray}
	(1-\bar{\mu})[\xi^T\zeta-\xi^T\eta+(m-1)(\hat{\mathcal{A}}\xi^{m-1})^T \xi]+\bar{\mu}\xi^T\xi=0\label{n1}\\
	\xi_i \eta_i=0 \ \forall \ i \label{n2}\\
	 (1-\bar{\mu})\xi_i\zeta_i=0 \ \forall \ i \label{n3}\\
	  \zeta-(1-\bar{\mu})\mathcal{A}\xi^{m-1}=0 \label{n4}
		\end{eqnarray}
		 From equation \ref{n4} $\zeta=(1-\bar{\mu})\mathcal{A}\xi^{m-1}$. Now multiplying $\xi^T $ in both sides we obtain $\xi^T \zeta=(1-\bar{\mu})\xi^T \mathcal{A}\xi^{m-1}=(1-\bar{\mu})\xi^T \hat{\mathcal{A}}\xi^{m-1} \implies  (\hat{\mathcal{A}}\xi^{m-1})^T \xi=\frac{1}{(1-\bar{\mu})} \zeta^T \xi$. Hence from the equation \ref{n1}, we obtain $(1-\bar{\mu})(m-1)(\hat{\mathcal{A}}\xi^{m-1})^T\xi+\bar{\mu}\xi^T\xi=0 \implies $  $\zeta^T\xi=-\frac{\bar{\mu}}{(m-1)}\xi^T\xi \leq 0.$   for $\bar{\mu} \in [0,1].$ \\ Specifically for $\bar{\mu}\in (0,1),$ $\zeta^T\xi<0,$ contradicts that $\xi, \zeta \geq 0.$ Hence the solution set  $\Gamma_v^{(0)}$ is bounded for $\bar{\mu}\in (0,1).$  Note that for $\bar{\mu}=0, \zeta^T\xi=0 \implies \xi=0$ and for $\bar{\mu}=1,$ $\xi=0.$  This contradicts that $e^T\xi=1.$ Therefore the solution set  $\Gamma_v^{(0)}$ is bounded for $\bar{\mu}\in [0,1].$   
	\vsp
		\NI	Subcase (ii) $\lim_{k\to \infty}(1-\mu^k)e^Tx^k= \infty:$\\
	 Let $\lim_{k\to \infty}\frac{(1-\mu^k)x^k}{(1-\mu^k)e^Tx^k}=\xi'\geq 0.$ Then $e^T\xi'=1.$  Let  $\lim_{k\to \infty}\frac{w^k}{(1-\mu^k)e^Tx^k}=\zeta' \geq 0, \ \lim_{k\to \infty}\frac{z_1^k}{(1-\mu^k)e^Tx^k}=\eta' \geq 0.$ Then multiplying the equations \ref{zq}, \ref{yyq}  with $(1-\mu^k)$ and dividing by $((1-\mu^k)e^Tx^k)^2$ and  dividing the equations \ref{xxq} and \ref{wwq} by $(1-\mu^k)e^Tx^k$ and $((1-\mu^k)e^Tx^k)^2$ respectively and taking limit $k \to \infty$, we write
		\begin{eqnarray}
		(1-\bar{\mu})[(\xi')^T\zeta'-(\xi')^T\eta']+(m-1)(\hat{\mathcal{A}}{\xi'}^{m-1})^T \xi'+\frac{\bar{\mu}}{1-\bar{\mu}}(\xi')^T\xi' =0 \label{n11}\\
	\xi'_i \eta'_i=0 \ \forall \ i \label{n22}\\
	 \xi'_i\zeta'_i=0 \ \forall \ i \label{n33}\\
	  \zeta'-\mathcal{A}{\xi'}^{m-1}=0 \label{n44}
		\end{eqnarray}
Using equations \ref{n22}, \ref{n33} and \ref{n44} in the equation \ref{n11}, we obtain\\ $(\zeta')^T\xi'=-\frac{\bar{\mu}}{(m-1)(\bar{\mu})}(\xi')^T\xi'\leq 0$.\\ Specifically for $\bar{\mu}\in (0,1),$ $\zeta^T\xi<0,$ contradicts that $\xi, \zeta \geq 0.$ Hence the solution set  $\Gamma_v^{(0)}$ is bounded for $\bar{\mu}\in (0,1).$  Note that for $\bar{\mu}=0, \zeta^T\xi=0\implies \xi=0$ and for $\bar{\mu}=1,$ $\xi=0.$  This contradicts that $e^T\xi=1.$ Therefore the solution set  $\Gamma_v^{(0)}$ is bounded for $\bar{\mu}\in [0,1].$

	\vsp
	\NI Case 2:  $\lim_{k\to \infty}e^Tz_2^k=\infty$. Since the solution set of $\Gamma_v^{(0)}$ is unbounded we consider following two subcases.\\ \NI Subcase (i) $\lim_{k\to \infty}e^Tx^k=\infty$: \\	
		Let $\lim\limits_{k\to \infty}\frac{x^k}{e^Tx^k}=\xi \geq 0, $ $\lim\limits_{k\to \infty}\frac{w^k}{e^Tx^k}=\zeta \geq 0, \ \lim\limits_{k\to \infty}\frac{z_1 ^k}{e^Tx^k}=\eta \geq 0, \ \lim\limits_{k\to \infty}\frac{z_2 ^k}{e^Tx^k}=\sigma \geq 0. $ It is clear that $e^T\xi=1.$ Then dividing by $e^Tx^k$ and taking limit $k \to \infty $ from equation \ref{xxq} and dividing by $(e^Tx^k)^2$ and taking limit $k \to \infty $ from equations \ref{zq}, \ref{yyq} and \ref{wwq}, we write
\begin{eqnarray}
	(1-\bar{\mu})[(\xi)^T\zeta-(\xi)^T\eta+(m-1)(\hat{\mathcal{A}}\xi^{m-1})^T (\xi-\sigma)]+\bar{\mu}(\xi)^T\xi =0 \label{nn1}\\
	\xi_i \eta_i=0 \ \forall \ i \label{nn2}\\
	 \sigma_i\zeta_i+(1-\bar{\mu})\xi_i\zeta_i=0 \ \forall \ i \label{nn3}\\
	  \zeta-(1-\bar{\mu})\mathcal{A}\xi^{m-1}=0 \label{nn4}
		\end{eqnarray}
		Using the equations \ref{nn2}, \ref{nn3} and \ref{nn4}, from equation \ref{nn1}, we obtain\\
		$(m-1)\zeta^T(\xi-\sigma)-\zeta^T\sigma=-\bar{\mu}(\xi)^T\xi \implies (m-1)\zeta^T\xi+(m-1)(1-\bar{\mu})\zeta^T\xi+((1-\bar{\mu}))\zeta^T\xi=-\bar{\mu}(\xi)^T\xi 
		\implies \zeta^T\xi=-\frac{\bar{\mu}}{m-1+m(1-\bar{\mu})}(\xi)^T\xi \leq 0.$
		Specifically for $\bar{\mu}\in (0,1),$ $\zeta^T\xi<0,$ contradicts that $\xi, \zeta \geq 0.$ Hence the solution set  $\Gamma_v^{(0)}$ is bounded for $\bar{\mu}\in (0,1).$  Note that for $\bar{\mu}=0, \zeta^T\xi=0\implies \xi=0$ and for $\bar{\mu}=1,$ $\xi=0.$  This contradicts that $e^T\xi=1.$ Therefore the solution set  $\Gamma_v^{(0)}$ is bounded for $\bar{\mu}\in [0,1].$ \\    
	\vsp
\NI	Subcase(ii) $\lim_{k\to \infty}(1-\mu^k)e^Tx^k= \infty:$\\
	 Let $\lim_{k\to \infty}\frac{(1-\mu^k)x^k}{(1-\mu^k)e^Tx^k}=\xi'\geq 0.$ Then $e^T\xi'=1.$  Let  $\lim_{k\to \infty}\frac{w^k}{(1-\mu^k)e^Tx^k}=\zeta' \geq 0, \ \lim_{k\to \infty}\frac{z_1^k}{(1-\mu^k)e^Tx^k}=\eta' \geq 0, lim_{k\to \infty}\frac{z_2^k}{(1-\mu^k)e^Tx^k}=\sigma' \geq 0.$ Then multiplying the equations \ref{zq}, \ref{yyq}  with $(1-\mu^k)$ and dividing by $((1-\mu^k)e^Tx^k)^2$ and  dividing the equations \ref{xxq} and \ref{wwq} by $(1-\mu^k)e^Tx^k$ and $((1-\mu^k)e^Tx^k)^2$ respectively and taking limit $k \to \infty$,  we obtain
	\begin{eqnarray}
		(1-\bar{\mu})[(\xi')^T\zeta'-(\xi')^T\eta']+(m-1)(\hat{\mathcal{A}}{\xi'}^{m-1})^T (\xi'-\sigma')+\frac{\bar{\mu}}{1-\bar{\mu}}(\xi')^T\xi' =0 \label{nn11}\\
	\xi'_i \eta'_i=0 \ \forall \ i \label{nn22}\\
	 \sigma'_i\zeta'_i + \xi'_i\zeta'_i=0 \ \forall \ i \label{nn33}\\
	  \zeta'-\mathcal{A}{\xi'}^{m-1}=0 \label{nn44}
		\end{eqnarray}
			Using the equations \ref{nn22}, \ref{nn33} and \ref{nn44}, from equation \ref{nn11}, we obtain\\
		$	(m-1)(\zeta')^T(\xi'-\sigma')+(1-\bar{\mu})\xi'^T\zeta'=-\frac{\bar{\mu}}{1-\bar{\mu}}(\xi')^T\xi' \implies (m-1)((\zeta')^T\xi'+(\zeta')^T\xi')+(1-\bar{\mu})\xi'^T\zeta'=-\frac{\bar{\mu}}{1-\bar{\mu}}(\xi')^T\xi' \implies \zeta'^T\xi'= -\frac{\bar{\mu}}{2(m-1)(1-\bar{\mu})^2}(\xi')^T\xi'\leq 0,$ contradicts that $\xi, \zeta \geq 0.$ Hence the solution set  $\Gamma_v^{(0)}$ is bounded for $\bar{\mu}\in (0,1).$  Note that for $\bar{\mu}=0, \zeta^T\xi=0\implies \xi=0$ and for $\bar{\mu}=1,$ $\xi=0.$  This contradicts that $e^T\xi=1.$ Therefore the solution set  $\Gamma_v^{(0)}$ is bounded for $\bar{\mu}\in [0,1].$ \\  
		 
		\end{proof}
 Hence considering all the cases it is proved that  the solution set $\Gamma_v^{(0)}$ of the homotopy function \ref{homf}, $H(v,v^{(0)},\mu)=0$  is bounded for $\bar{\mu} \in [0,1]$.
For an initial point $v^{(0)}\in \mathcal{F}_1$ we obtain a smooth bounded homotopy path which  leads to  the solution of homotopy function \ref{homf} as the parameter $\mu \to 0.$

\begin{theorem}\label{3.7}
	For $v^{(0)}=(x^{(0)},w^{(0)},z_1^{(0)},z_2^{(0)})\in \mathcal{F}_1,$ the homotopy continuation method finds a bounded smooth curve $\Gamma_v^{(0)} \subset \mathcal{F}_1 \times (0,1]$ which starts from $(v^{(0)},1)$ and approaches the hyperplane at $\mu =0.$ As $\mu\to 0,$ the limit set $l \times \{0\} \subset \bar{\mathcal{F}}_1 \times \{0\}$ of $\Gamma_v^{(0)}$ is nonempty and every point in $l$ is a solution of the following system:
	\begin{equation}\label{sys}
	\begin{split}
	w-z_1+(m-1)(\hat{\mathcal{A}}x^{m-2})^T(x-z_2)=0 \\
 Z_1x=0\\
 Z_2w+Xw=0\\
 w-(\mathcal{A}x^{m-1} + q)=0\\
 	\end{split}
    \end{equation}
\end{theorem}
\begin{proof}
  	Note that in view of lemma \ref{cl}, $\Gamma_v^{(0)}$ is diffeomorphic to a unit circle or a unit interval $(0,1].$  Since $\frac{\partial{H(v,v^{(0)},1)}}{\partial{v^{(0)}}}$ is nonsingular, $\Gamma_v^{(0)}$ is diffeomorphic to a unit interval $(0,1].$ $\Gamma_v^{(0)}$ is a bounded smooth curve by the theorem \ref{reg} and \ref{001}. Let $(\bar{v},\bar{\mu})$ be a limit point of $\Gamma_v^{(0)}.$ Now we consider the following four cases: \\$(i)(\bar{v},\bar{\mu})\in \mathcal{F}_1 \times \{1\}:$ As the homotopy function $H(v,1)=0$ has only one solution $v^{(0)}\in  \mathcal{F}_1, $ this case is impossible.\\
  	$(ii)(\bar{v},\bar{\mu})\in \partial{\mathcal{F}_1} \times \{1\}:$ There exists a subsequence of $(v^k, \mu^k) \in \Gamma_v^{(0)}$ such that $x_i^k \to 0$ or $w_i ^k \to 0$ for $i \subseteq \{1,2,\cdots n\}.$ From the second and third equalities of the homotopy function \ref{homf}, we have $z_1^k \to \infty$ or $z_2^k \to \infty.$ Hence it contradicts the boundedness of the homotopy path by the theorem \ref{001}.\\
  	$(iii)(\bar{v},\bar{\mu})\in \partial{\mathcal{F}_1} \times (0,1):$ Also impossible followed by the case $(ii).$\\
  	$(iv)(\bar{v},\bar{\mu})\in \bar{\mathcal{F}}_1 \times \{0\}:$ The only possible case.\\

 \NI Hence $\bar{v}=(\bar{x},\bar{w},\bar{z_1},\bar{z_2})$ is a solution of the system \ref{sys}
 \begin{center}
    $	w-z_1+(m-1)(\hat{\mathcal{A}}x^{m-2})^T(x-z_2)=0 $\\
 $Z_1x=0$\\
 $Z_2w+Xw=0$\\
$w-(\mathcal{A}x^{m-1} + q)=0$\\
 \end{center}
 \end{proof}

\begin{remk}\label{mmf}
	So, from the homotopy function \ref{homf}  as $\mu \to 0$ we get $\bar{w}-\bar{z}_1$ $+(m-1)(\hat{\mathcal{A}}{\bar{x}}^{(m-2)})^T (\bar{x}-\bar{z}_2)=0,$ 
	$\bar{w}=\mathcal{A}{\bar{x}}^{(m-1)}+q$ and
$\bar{z}_{1i}\bar{x}_i=0$, $\bar{z}_{2i}\bar{w}_i=0, \bar{x}_{i}\bar{w}_i=0 \ \forall i\in\{1,2, \cdots n\}$. Hence $\mu \to 0$ we get the solution of the homotopy function as well as  TCP$(\mathcal{A},q)$.
\end{remk}

 \begin{remk}
	We find the homotopy path $\Gamma_v^{(0)} \subset \mathcal{{F}}_{(m_1)} \times (0,1]$ from the initial point $(v^{(0)},1)$ until $\mu \to 0$ and find the solution of the given complementarity problem \ref{TCP} under some assumptions. Let $s$ denote the arc length of $\Gamma_v^{(0)}.$  We can parameterize the homotopy path $\Gamma_v^{(0)}$ for $s$ in the following form\\
	\begin{equation}\label{ss}
	H (v(s),\mu (s))=0, \ 
	v(0)=v^{(0)}, \   \mu(0)=1.
	\end{equation} 
	Now differentiating \ref{ss} with respect to $s$ we obtain the following system of ordinary differential equations with given initial values\cite{fan}
	\begin{equation}
	H' (v(s),\mu (s))\left[\begin{array}{c} 
	\frac{dv}{ds}\\
	\frac{d\mu}{ds}\\
	\end{array}\right]=0, \
	\|( \frac{dv}{ds},\frac{d\mu}{ds})\|=1, \ 
	v(0)=v^{(0)}, \   \mu(0)=1, \ \frac{d\mu}{ds}(0)<0, 
	\end{equation} 
	and the $v$-component of $(v(\bar{s}),\mu (\bar{s}))$ gives the solution of the complementarity problem for $\mu (\bar{s})=0.$
\end{remk}

Now we use the modified homotopy continuation method to trace the homotopy path $\Gamma_v^{(0)}$ numerically to find the solution of TCP$(q,\mathcal{A})$. 

\begin{algorithm}
\caption{ \ \textbf{\textit{Modified Homotopy Continuation Method}}}\label{allgo}
%\begin{algorithmic}
%\floatname{algorithm}{modified homotopy continuation method}
%\begin{algorithmic}

%\SetAlgorithmName{ \textbf{Algorithm}(Modified Homotopy Continuation Method)}
\NI \textbf{Step 0:} Initialize $(v^{(0)},\mu_0)$ and a natural number $m\in(0,50).$ Set $l_0 \in (0, 1).$ Choose $\epsilon_2 >> \epsilon_3 >> \epsilon_1 > 0$ which are very small positive values.
\vsp 
\NI \textbf{Step 1:} $\tau^{(0)}= \xi^{(0)}=(\frac{1}{n_0})\left[\begin{array}{c} 
s^{(0)}\\
-1\\
\end{array}\right]$ for $i=0,$ where $n_0=\|\left[\begin{array}{c} 
s^{(0)}\\
-1\\
\end{array}\right]\|$ and $s^{(0)}= (\frac{\partial H}{\partial v}(v^{(0)},\mu_0))^{-1}(\frac{\partial H}{\partial \mu}(v^{(0)},\mu_0)).$ \\ For $i>0,$ $s^{(i)}= (\frac{\partial H}{\partial v}(v^{(i)},\mu_i))^{-1}(\frac{\partial H}{\partial \mu}(v^{(i)},\mu_i)),$ $n_i=\|\left[\begin{array}{c} 
s^{(i)}\\
-1\\
\end{array}\right]\|,$ $\xi^{(i)}=(\frac{1}{n_i})\left[\begin{array}{c} 
s^{(i)}\\
-1\\
\end{array}\right]$.\\ 
If $\det  (\frac{\partial H}{\partial v}(v^{(i)},\mu_i))>0,$ $\tau^{(i)}= \xi^{(i)}$ else $\tau^{(i)}= -\xi^{(i)},$ $i \geq 1.$\\ Set $l=0.$
\vsp
\NI \textbf{Step 2:} (Predictor and corrector point calculation) $(\tilde{v}^{(i)},\tilde{\mu}_i)=(v^{(i)},\mu_i)+a\tau^{(i)},$ where $a={l_0}^l.$ Compute $(\hat{v}^{(i)},\hat{\mu}_{i})=H'_{v^{(0)}}(\tilde{v}^{(i)},\tilde{\mu}_i)^+ H(\tilde{v}^{(i)}, \tilde{\mu}_i)$ and $(\bar{v}^{(i)},\bar{\mu}_{i})=(\tilde{v}^{(i)},\tilde{\mu}_i)-(\hat{v}^{(i)},\hat{\mu}_{i}).$ Now compute $(\hat{vv}^{(i)},\hat{\mu\mu}_{i})=(H'_{v^{(0)}}(\tilde{v}^{(i)},\tilde{\mu}_i)+H'_{v^{(0)}}(\bar{v}^{(i)},\bar{\mu}_i))^+ H(\tilde{v}^{(i)}, \tilde{\mu}_i)$ and $(\bar{vv}^{(i)},\bar{\mu\mu}_{i})=(\tilde{v}^{(i)},\tilde{\mu}_i)-2(\hat{vv}^{(i)},\hat{\mu\mu}_{i}).$ \\Compute $(v^{(i+1)},\mu_{i+1})=(\bar{vv}^{(i)},\bar{\mu\mu}_{i})-H'_{v^{(0)}}(\bar{v}^{(i)},\bar{\mu}_i)^+H(\bar{vv}^{(i)},\bar{\mu\mu}_{i}).$ \\Continue the method from the computation of $(\tilde{v}^{(i)},\tilde{\mu}_{i})$ to the computation of $(v^{(i+1)},\mu_{i+1})$ for $m$ times. In each step after repeating the computation for $m$ times, can obtain the value for next iteration $(v^{(i+1)},\mu_{i+1})$.  \\ If $0<\|{\mu_{i+1}-\mu_i}\|<1, $ go to step 3. Otherwise if $m' = \min(a,\|({v}^{(i+1)},{\mu}_{i+1})-({v}^{(i)},{\mu}_{i})\|)>a_0,$ update $l$ by $l+1,$ and recompute $(\tilde{\mu}_i, \hat{\mu}_{i})$ else go to step 3.
\vsp
\NI \textbf{Step 3:} Determine the norm $r=\|H(v^{(i+1)},\mu_{i+1})\|.$ If $r \leq 1$ and $v^{(i+1)}>0$ go to step 5, otherwise if $a > \epsilon_3,$ update $l$ by $l+1$ and go to step 2 else go to step 4.
\vsp
\NI \textbf{Step 4:} If $|\mu_{i+1} - \mu_i| < \epsilon_2,$ then if $|\mu_{i+1}| < \epsilon_2,$ then stop with the solution $(v^{(i+1)},\mu_{i+1}),$ else terminate (unable to find solution) else $i=i+1$ and go to step 1.
\vsp
\NI \textbf{Step 5:} If $|\mu_{i+1}| \leq \epsilon_1,$ then stop with solution $(v^{(i+1)},\mu_{i+1}),$ else $i=i+1$ and go to step 1.
%\end{algorithmic}
\end{algorithm}

\NI Note that in step 2, $H'_{v^{(0)}}(v,\mu)^+ = H'_{v^{(0)}}(v,{\mu})^{T}(H'_{v^{(0)}}(v,{\mu})H'_{v^{(0)}}(v,\mu)^{T})^{-1}$ is the Moore-Penrose inverse of $H'_{v^{(0)}}(v,\mu).$ 
\vsp
 
\NI Now we obtain the sign of the positive tangent direction of the homotopy path.

\begin{theorem}
	If the homotopy curve $\Gamma_v^{(0)}$ is smooth, then the positive predictor direction $\tau^{(0)}$ at the initial point $v^{(0)}$ satisfies sign($\det \left[\begin{array}{c}
	\frac{\partial H}{\partial v \partial \mu}(v,v^{(0)},1)\\
	\tau^{(0)^T}\\
	\end{array}\right]$)$<0.$ 
\end{theorem}
\begin{proof}
	From equation \ref{homf} we have
	
	$H(v,v^{(0)},\mu)=\\$
	$\left[\begin{array}{c} 
 (1-\mu)(w-z_1+(m-1)(\hat{\mathcal{A}}x^{m-2})^T(x-z_2))+\mu(x-x^{(0)}) \\
 Z_1x-\mu Z_1^{(0)}x^{(0)}\\
 Z_2w-\mu Z_2^{(0)}w^{(0)}+(1-\mu)Xw\\
 w-(1-\mu)(\mathcal{A}x^{m-1} + q)-\mu w^{(0)}\\
  \end{array}\right]=0.$
 \vsp
	 Now at the initial point $(v^{(0)},1)$ the patial derivative of the homotopy function \ref{homf} is given by, 	$\frac{\partial H}{\partial v \partial \mu}(v,\mu)=\begin{bmatrix}
	L_5 & L_6
	\end{bmatrix},$ where\\
	$L_5=\begin{bmatrix} I & 0 & 0 & 0 \\
	Z_1^{(0)} & 0 & X^{(0)} & 0 \\
	0 &  Z_2 ^{(0)} & 0 & W^{(0)} \\ 
	0 & I & 0 & 0 \\
	 \end{bmatrix}.$\\  $L_6=\begin{bmatrix}
	A_1\\
	B_1\\
	C_1\\
	D_1\\
		\end{bmatrix},$ $X^{(0)}=\text{diag}(x^{(0)}), W^{(0)}=\text{diag}(w^{(0)})$, $Z_1^{(0)}=\text{diag}(z_1^{(0)})$, $Z_2^{(0)}=\text{diag}(z_2^{(0)}),\\A_1=(x-x^{(0)})-[w-z_1+(m-1)(\hat{\mathcal{A}}x^{m-2})^T(x-z_2)], B_1=-Z_1^{(0)}x^{(0)}, C_1=-Z_2^{(0)}w^{(0)}-X^{(0)}w^{(0)}, D_1= \mathcal{A}x^{m-1}-w^{(0)}. $\\
	Let positive predictor direction be $\tau^{(0)}=\left[\begin{array}{c}
	t \\ -1
	\end{array}\right]=\left[\begin{array}{c}
	(\mathbb{R}^{(0)}_1)^{(-1)}R_2^{(0)} \\ -1
	\end{array}\right],$\\ where $\mathbb{R}^{(0)}_1=
\begin{bmatrix} 
   I & 0 & 0 & 0 \\
	Z_1^{(0)} & 0 & X^{(0)} & 0 \\
	0 &  Z_2 ^{(0)} & 0 & W^{(0)} \\ 
	0 & I & 0 & 0 \\
	 \end{bmatrix}$\\ 
	  and $\mathbb{R}^{(0)}_2=\left[\begin{array}{c} 
	A_1\\
	B_1\\
	C_1\\
	D_1\\
		\end{array}\right],$ where $A_1=(x-x^{(0)})-[w-z_1+(m-1)(\hat{\mathcal{A}}x^{m-2})^T(x-z_2)], B_1=-Z_1^{(0)}x^{(0)}, C_1=-Z_2^{(0)}w^{(0)}-Xw, D_1= \mathcal{A}x^{m-1}-w^{(0)}.$\\
	Here  \ $\text{det}(\mathbb{R}^{(0)}_1)=	\prod_{i=1}^{n} x_i^{(0)}w_i^{(0)}> 0.$ \\
	\vsp
	Therefore, 
	$\det\left[\begin{array}{c}
	\frac{\partial H}{\partial v \partial \mu}(v,v^{(0)},1)\\
	\tau ^{(0)^T}\\
	\end{array}\right]$ $=\det\left[\begin{array}{cc}
	\mathbb{R}^{(0)}_1 & \mathbb{R}^{(0)}_2\\
	(\mathbb{R}^{(0)}_2)^T(\mathbb{R}^{(0)}_1)^{(-t)} & -1\\	
	\end{array}\right]$ \\ $= \det\left[\begin{array}{cc}
	\mathbb{R}^{(0)}_1 & \mathbb{R}^{(0)}_2\\
	0 & -1-(\mathbb{R}^{(0)}_2)^T(\mathbb{R}^{(0)}_1)^{(-t)}(\mathbb{R}^{(0)}_1)^{(-1)}R_2^{(0)} \\	
	\end{array}\right] \\ =\det(\mathbb{R}^{(0)}_1)\det(-1-(\mathbb{R}^{(0)}_2)^T(\mathbb{R}^{(0)}_1)^{(-t)}(\mathbb{R}^{(0)}_1)^{(-1)}R_2^{(0)}) \\ =-\det(\mathbb{R}^{(0)}_1)\det(1+(\mathbb{R}^{(0)}_2)^T(\mathbb{R}^{(0)}_1)^{(-t)}(\mathbb{R}^{(0)}_1)^{(-1)}R_2^{(0)}) \\ =\prod_{i=1}^{n} x_i^{(0)}w_i^{(0)}\det(1+(\mathbb{R}^{(0)}_2)^T(\mathbb{R}^{(0)}_1)^{(-t)}(\mathbb{R}^{(0)}_1)^{(-1)}R_2^{(0)})<0. $ 
\end{proof}
\section{Numerical Examples}
\begin{examp}
Let consider a column sufficient tensor $\mathcal{A}\in T_{4,2}$ sucth that $a_{1112}=-2, a_{2111}=1, a_{2222}=1.$ Other entries are zero. $q=\left[\begin{array}{c}
	1\\
  -1\\	
	\end{array}\right].$ The initial point is $v^{(0)}=(1,1,1,1,1,1,1,1,1).$ After $10$ iterations the solution of the tensor complementarity problem is $z=\left[\begin{array}{c}
	.7937\\
  .7937\\	
	\end{array}\right]$ and $w=\left[\begin{array}{c}
	0\\
  0\\	
	\end{array}\right].$
	\end{examp}
	
	\begin{examp}
	Let consider a column competent tensor $\mathcal{A}\in T_{3,2}$ sucth that $a_{112}=1, a_{121}=1, a_{122}=1, a_{212}=1, a_{221}=1, a_{222}=1.$ Other entries are zero. $q=\left[\begin{array}{c}
	-2\\
  1\\	
	\end{array}\right].$ The initial point is $v^{(0)}=(1,1,1,1,1,1,1,1,1).$ After $51$ iterations the solution of the tensor complementarity problem is $z=\left[\begin{array}{c}
	1697.278\\
  0\\	
	\end{array}\right]$ and $w=\left[\begin{array}{c}
	0\\
  3\\	
	\end{array}\right].$
	\end{examp}
	\begin{examp}
		Let consider a tensor $\mathcal{A}\in T_{4,2}$ such that $a_{1111}=1, a_{1112}=-1, a_{2111}=1.$ All other entries are zero. $q=\left[\begin{array}{c}
	1\\
  -1\\	
	\end{array}\right].$ It is column adequate tensor.
	The initial point is $v^{(0)}=(1,1,1,1,1,1,1,1,1).$ After $51$ iterations the solution reaches to $z=\left[\begin{array}{c}
	1\\
  2\\	
	\end{array}\right]$ and $w=\left[\begin{array}{c}
	0\\
  0\\	
	\end{array}\right].$
	\end{examp}
	\begin{examp}
		Let consider a $P_0$-tensor $\mathcal{A}\in T_{4,2}$ sucth that $a_{1111}=2, a_{1112}=1, a_{2122}=4, a_{2222}=2.$ All other entries are zero. $q=\left[\begin{array}{c}
	-1\\
  -1\\	
	\end{array}\right].$  The initial point is $v^{(0)}=(1,1,1,1,1,1,1,1,1).$ After $11$ iterations the solution of the tensor complementarity problem is $z=\left[\begin{array}{c}
	0.717516\\
  0.50706\\	
	\end{array}\right]$ and $w=\left[\begin{array}{c}
	0\\
  0\\	
	\end{array}\right].$
	\end{examp}
		\begin{examp}
		Let consider a strictly semi positive tensor $\mathcal{A}\in T_{3,2}$ sucth that $a_{111}=1, a_{121}=2, a_{122}=1, a_{222}=1, a_{211}=-1, a_{221}=-1.$ All other entries are zero. But this tensor is not strong strictly semi positive. $q=\left[\begin{array}{c}
	-1.5\\
  1\\	
	\end{array}\right].$  The initial point is $v^{(0)}=(1,1,1,1,1,1,1,1,1).$ After $16$ iterations the solution  is $z=\left[\begin{array}{c}
	0.901703\\
  0.3230419\\	
	\end{array}\right]$ and $w=\left[\begin{array}{c}
	0\\
  0\\	
	\end{array}\right].$
	\end{examp}
	\begin{examp}
		Let consider a  semi positive tensor $\mathcal{A}\in T_{3,2}$ sucth that $a_{111}=1, a_{112}=-3, a_{122}=1, a_{222}=1, a_{211}=1, a_{212}=-2.$ All other entries are zero. But this tensor is neither  strictly semi positive nor strong strictly semipositive . $q=\left[\begin{array}{c}
	-2\\
  -1\\	
	\end{array}\right].$  The initial point is $v^{(0)}=(1,1,1,1,1,1,1,1,1).$ After $19$ iterations the solution of the tensor complementarity problem reaches to $z=\left[\begin{array}{c}
	1.414214\\
  0\\	
	\end{array}\right]$ and $w=\left[\begin{array}{c}
	0\\
  1\\	
	\end{array}\right].$
	\end{examp}

\section{Conclusion}
In this study we introduce a new homotopy function to process the tensor complementarity problem under some conditions.  We show that the proposed function is smooth and the homotopy path is bounded. We prove necessary and sufficient conditions to obtain the solution of tensor complementarity problem from the solution of the homotopy equation. In this connection we obtain the sign of the positive tangent direction of the homotopy path. We construct a homotopy continuation method and show that the method reaches to solution. Finally, several numerical illustrations are presented to show the effectiveness of the proposed approach.

\section{Acknowledgment}
The author A. Dutta is thankful to the Department of Science and Technology, Govt. of India, INSPIRE Fellowship Scheme for financial support. The author Bharat Kumar is thankful to the University Grant Commission (UGC),  Government of India under the JRF programme no.  1068/(CSIR-UGC NET DEC. 2017 ).

\bibliographystyle{plain}
\bibliography{lisi}
\end{document}